\documentclass[hidelinks,12pt,a4paper,reqno]{amsart}
\usepackage[utf8]{inputenc}
\usepackage{enumitem}
\usepackage{amsmath}
\usepackage{amsfonts}
\usepackage{amssymb}
\usepackage[all]{xy}
\usepackage{xfrac}
\usepackage[yyyymmdd,hhmmss]{datetime}
\usepackage{centernot}
\usepackage{comment,xcolor}
\usepackage[english]{babel}
\usepackage{dsfont}
\usepackage{bbold}
\usepackage{soul}

\usepackage[backref=page]{hyperref}
\renewcommand*{\backref}[1]{}
\renewcommand*{\backrefalt}[4]{%
    \ifcase #1 (Not cited.)%
    \or        (Cited on page~#2.)%
    \else      (Cited on pages~#2.)%
    \fi}

\newtheorem{theorem}{Theorem}[section]

\newtheorem{definition}{Definition}[section]
\newtheorem{proposition}[theorem]{Proposition}

\newcommand{\R}{\mathbb{R}}

\DeclareMathOperator{\End}{End}

\newcommand{\muu}{\text{\sfrac{1}{2}}}

\newcommand{\oh}{\frac{1}{2}}
\DeclareMathOperator{\arccosh}{arccosh}

 \author{J. P. Fatelo and N. Martins-Ferreira}
 \address{School of Technology and Management, Centre for Rapid and Sustainable Product Development - CDRSP, Polytechnic Institute of Leiria, P-2411-901 Leiria, Portugal.}
 \email{martins.ferreira@ipleiria.pt}
\title[]{Mobility spaces and geodesics for the n-sphere}


\thanks{  This work is supported by the Fundação para a Ciência e a Tecnologia (FCT) and Centro2020 through the Project references: UID/Multi/04044/2019; PAMI - ROTEIRO/0328/2013 (Nº 022158); Next.parts (17963), and also by CDRSP and ESTG from the Polytechnic Institute of Leiria.
}

\begin{document}

\begin{abstract}

We introduce an algebraic system which can be used as a model for spaces
with geodesic paths between any two of their points. This new algebraic structure is based on the notion of mobility algebra which has recently been introduced as a model for the unit interval of real numbers. We show that there is a strong connection between modules over a ring and affine mobility spaces over a mobility algebra. However, geodesics in general fail to be affine thus giving rise to the new algebraic structure of mobility space. We show that the so called formula for spherical linear interpolation, which gives geodesics on the n-sphere, is an example of a mobility space over the unit interval mobility algebra. 

\end{abstract}

\maketitle

\today; \currenttime

\section{Introduction}

The purpose of this work is to introduce an algebraic system which can be used to model spaces with geodesics. The main idea stems from the interplay between algebra and geometry. In affine geometry the notion of affine space is well suited for this purpose. Indeed, in an affine space we have scalar multiplication, addition and subtraction and so it is possible to parametrize, for any instant $t\in [0,1]$, a straight line between points $x$ and $y$ with the formula $(1-t)x+ty$. Such a line is clearly a geodesic path from $x$ to $y$. 
In general terms, we may use an operation $q=q(x,t,y)$ to indicate the position, at an instant~$t$, of a particle moving in a space $X$ from a point $x$ to a point $y$. If the particle is moving along a geodesic path then this operation must certainly verify some conditions. The aim of this project is to present an algebraic structure, $(X,q)$, with axioms that are verified by any operation~$q$ representing a geodesic path in  a space between any two of its points. 

The first results concerning this investigation were presented in \cite{ccm_magmas} and~\cite{mobi}. In \cite{preprint}, the first version of this text, mobility algebras (or mobi algebras) and mobility affine spaces are considered in more details. The remaining study of mobility spaces (mobi space for short) already presented in \cite{preprint} has been reformulated and further developed here. 

In the preprint \cite{preprint}, it is shown that every space with unique geodesics can be given a mobi space structure. However, when geodesics are not unique (for instance when connecting antipodal points on the sphere) it is still possible to define a mobi space structure on that space. This is done by making appropriate choices and it is illustrated in the last section of this paper. The example of the sphere is considered with spherical linear interpolation (Slerp) whose formula gives rise to a mobi space structure.

\section{Mobi algebra}\label{sec_mobi_alg}

In this section we briefly recall the notion of mobi algebra, introduced in \cite{mobi}, and some of its basic properties.
Several examples  of different nature were presented in \cite{mobi} and \cite{preprint}. 
As we will see in the next section (while introducing the notion of mobi space) a mobility algebra plays the role of scalars (for a mobility space) in the same way as a ring (or a field) models the scalars for a module (or a vector space) over the base ring (or field). 

 In order to have an intuitive interpretation of its axioms we may consider a mobi algebra as a mobi space over itself and use the geometric intuition provided in section \ref{sec_mobi_space}.
Namely, that the operation $p(x,t,y)$ is the position of a particle moving from a point $x$ to a point $y$ at an instant $t$ while following a geodesic path.

\begin{definition}\cite{mobi}\label{mobi_algebra}
A mobi algebra is a system $(A,p,0,\muu,1)$, in which $A$ is a set, $p$ is a ternary operation and $0$, $\muu$ and $1$ are elements of A, that satisfies the following axioms:
\begin{enumerate}[label={\bf (A\arabic*)}]
\item\label{alg_mu} $p(1,\muu,0)=\muu$
\item\label{alg_01} $p(0,a,1)=a$
\item\label{alg_idem} $p(a,b,a)=a$
\item\label{alg_0} $p(a,0,b)=a$
\item\label{alg_1} $p(a,1,b)=b$
\item\label{alg_cancel} $p(a,\muu,b_1)=p(a,\muu,b_2)\implies b_1=b_2$
\item\label{alg_homo} $p(a,p(c_1,c_2,c_3),b)=p(p(a,c_1,b),c_2,p(a,c_3,b))$
\item\label{alg_medial}
$p(p(a_1,c,b_1),\muu,p(a_2,c,b_2))=p(p(a_1,\muu,a_2),c,p(b_1,\muu,b_2))$.
\end{enumerate}
\end{definition}

Some properties of mobi algebras can be suitably expressed in terms of a unary operation~"$\overline{()}$" and binary operations~"$\cdot$", "$\circ$" and "$\oplus$" defined as follows (see \cite{mobi} for more details).
\begin{definition}\label{binary_definition} Let $(A,p,0,\muu,1)$ be a mobi algebra. We define:
\begin{eqnarray}
\label{def_complementar}\overline{a}&=&p(1,a,0)\\
\label{def_product}a\cdot b&=&p(0,a,b)\\
\label{def_star}a\oplus b&=&p(a,\muu,b)\\
\label{def_oplus}a\circ b&=&p(a,b,1).
\end{eqnarray}
\end{definition}

Some properties of a mobi algebra follow. If $(A,p,0,\muu,1)$ is a mobi algebra, then: 
\begin{eqnarray}
\label{Bmuu} \overline{\muu}&=&\muu\\
\label{B130} a\cdot\muu=\muu\cdot a&=&0\oplus a\\
\label{cancell-muu} \muu\cdot a=\muu\cdot a'&\Rightarrow& a=a'\\
\label{B7} p(\overline{a},\muu,a)&=&\muu\\
\label{Boverlinea} \overline{a}=a &\Rightarrow & a=\muu\\
\label{complementary-p} \overline{p(a,b,c)}&=&p(\overline{a},b,\overline{c})\\
\label{commut} p(c,b,a)&=&p(a,\overline{b},c)\\
\label{circ_cdot} \overline{a\circ b}&=&\overline{b}\cdot\overline{a}\\
\label{important}\muu\cdot p(a,b,c)&=&(\overline{b}\cdot a)\oplus (b\cdot c).
\end{eqnarray}

We end this section recalling that if $(A,p,0,\muu,1)$ is a mobi algebra then $(A,\oplus)$, with $\oplus$ defined in (\ref{def_star}), is a midpoint algebra.
A midpoint algebra (\cite{Escardo}, see also \cite{Jezek})  consists of a set and a binary operation 
$\oplus$ satisfying the following axioms:
\begin{eqnarray}
 \text{(idempotency)} & x\oplus x= x\\
 \text{(commutativity)} & x\oplus y= y \oplus x\\
 \text{(cancellation)} & x\oplus y= x' \oplus y\Rightarrow x=x'\\
\label{medial}\text{(mediality)}&(x\oplus y)\oplus (z\oplus w)=( x \oplus z)\oplus (y\oplus w) .
\end{eqnarray}

\section{Mobi space}\label{sec_mobi_space}

In this section we give the definition of a mobi space over a mobi algebra. Its main purpose is to serve as a model for spaces with a geodesic path connecting any two points. It is similar to a module over a ring in the sense that it has an associated mobi algebra which behaves as the set of scalars. 
In Section \ref{sec:modules} we show that the particular case of affine mobi space is indeed the same as a module over a ring when the mobi algebra is a ring. In the last section, we present examples of geodesics on the $n$-sphere and on an hyperbolic $n$-space as mobility spaces over the unit interval.

\begin{definition}\label{mobi_space}
Let $(A,p,0,\muu,1)$ be a mobi algebra. An $A$-mobi space $(X,q)$, consists of a set $X$ and a map $q\colon{X\times A\times X\to X}$ such that:
\begin{enumerate}[label={\bf (X\arabic*)}]
\item\label{space_0} $q(x,0,y)=x$
\item\label{space_1} $q(x,1,y)=y$
\item\label{space_idem} $q(x,a,x)=x$
\item\label{space_cancel} $q(x,\muu,y_1)=q(x,\muu,y_2)\implies y_1=y_2$
\item\label{space_homo} $q(q(x,a,y),b,q(x,c,y))=q(x,p(a,b,c),y)$
\end{enumerate}
\end{definition}

The axioms \ref{space_0} to \ref{space_homo} are the natural generalizations of axioms \ref{alg_idem} to \ref{alg_homo} of a mobi algebra. The natural generalization of \ref{alg_medial} would be
\begin{eqnarray}\label{affine}
q(q(x_1,a,y_1),\muu,q(x_2,a,y_2))=q(q(x_1,\muu,x_2),a,q(y_1,\muu,y_2)).
\end{eqnarray}
This condition, however, is too restrictive and is not in general verified by geodesic paths.
That is the reason why we do not include it.
When (\ref{affine}) is satisfied for all $x_1,x_2,y_1,y_2\in X$ and $a\in A$, then we say that the $A$-mobi space $(X,q)$ is affine and speak of an $A$-mobi affine space (see \cite{preprint} and Subsection \ref{affineness}).

If we write $x\oplus y$ instead of $q(x,\muu,y)$ and consider the special case of  equation $(\ref{affine})$ when $a=\muu$ then we get the usual medial law (\ref{medial}). As an illustration of the fact that the medial law does not hold true in general for geodesic paths, let us consider the example of the unit sphere. The midpoint $a\oplus b$ of two points $a$ and $b$ on the equator is again on the equator. Midpoint of North Pole $n$ and any point $c$ on equator is on the $45^{\textrm{th}}$ parallel. But the geodesic midpoint of two points
on the $45^{\textrm{th}}$ parallel does not live on the $45^{\textrm{th}}$ parallel, but somewhat to the
North of it; the $45^{\textrm{th}}$ parallel is not a geodesic. So $(a\oplus b)\oplus (n\oplus n)$ is on the
$45^{\textrm{th}}$ parallel, but $(a\oplus n) \oplus (b \oplus n)$ is not, but is north of it. This phenomenon is an aspect of the Gaussian curvature of the sphere.

Every affine space $(X,+)$ over a commutative field of scalars can be considered as an affine mobi space $(X,q)$ with $q(x,t,y)=(1-t)x+ty$. At the end of this paper we give the formula for geodesics on the $n$-sphere in terms of a mobi space structure.

Here are some immediate consequences of the axioms for a mobi space.

\begin{proposition}\label{properties_space} Let $(A,p,0,\muu,1)$ be a mobi algebra and $(X,q)$ an A-mobi space. It follows that:
\begin{enumerate}[label={\bf (Y\arabic*)}]
\item\label{Y1} $q(y,a,x)=q(x,\overline{a},y)$
\item\label{Y2} $q(y,\muu,x)=q(x,\muu,y)$
\item\label{Y3} $q(x,a,q(x,b,y))=q(x,a\cdot b,y)$
\item\label{Y4} $q(q(x,a,y),b,y)=q(x,a\circ b,y)$
\item\label{Y7} $q(q(x,a,y),\muu,q(x,b,y))=q(x,a\oplus b,y)$
\item\label{Y6} $q(x,\muu,q(x,a,y))=q(x,a,q(x,\muu,y))$
\item\label{Y5} $q(q(x,a,y),\muu,q(y,a,x))=q(x,\muu,y)$
\item\label{Y8} $q(q(q(x,a,y),b,x),\muu,q(x,b,q(x,c,y)))\\
=q(x,\muu,q(x,p(a,b,c),y))$
\item\label{Y9} $q(x,a,y)=q(y,a,x)\Rightarrow q(x,a,y)=q(x,\muu,y)$
\item\label{Y10} $q(x,a,y)=q(x,b,y)\Rightarrow q(x,p(a,t,b),y)=q(x,a,y)$, \text{for all $t$}.
\end{enumerate}
\end{proposition}

\begin{proof}
The following proof of \ref{Y1}, bearing in mind (\ref{def_complementar}), uses \ref{space_homo}, \ref{space_1} and \ref{space_0}:
\begin{eqnarray*}
q(x,\overline{a},y)&=&q(x,p(1,a,0),y)\\
                   &=&q(q(x,1,y),a,q(x,0,y))\\
                   &=&q(y,a,x).
\end{eqnarray*}
\ref{Y2} follows directly from \ref{alg_mu} and \ref{Y1}.
Beginning with (\ref{def_product}), \ref{Y3} is a consequence of \ref{space_homo} and \ref{space_0}:
\begin{eqnarray*}
q(x,a\cdot b,y)&=&q(x,p(0,a,b),y)\\
          &=&q(q(x,0,y),a,q(x,b,y))\\
          &=&q(x,a,q(x,b,y)).
\end{eqnarray*}
Considering (\ref{circ_cdot}), property \ref{Y4} follows from \ref{Y1} and \ref{Y3}.
\begin{eqnarray*}
q(q(x,a,y),b,y)&=&q(y,\overline{b},q(y,\overline{a},x))\\
               &=&q(y,\overline{b}\cdot\overline{a},x)\\
               &=&q(y,\overline{a\circ b},x)\\
               &=&q(x,a\circ b,y).
\end{eqnarray*}
Considering $(\ref{def_oplus})$, \ref{Y7} is just a particular case of \ref{space_homo}.
To prove \ref{Y6}, we use \ref{space_homo}, (\ref{B130}) and \ref{space_0}.
\begin{eqnarray*}
q(x,\muu,q(x,a,y))&=&q(q(x,0,y),\muu,q(x,a,y))\\
                 &=&q(x,p(0,\muu,a),y)\\
                 &=&q(x,p(0,a,\muu),y)\\
                 &=&q(q(x,0,y),a,q(x,\muu,y))\\
                 &=&q(x,a,q(x,\muu,y)).
\end{eqnarray*}
The following proof of \ref{Y5} is based on \ref{Y1}, \ref{space_homo} and (\ref{B7});
\begin{eqnarray*}
q(q(x,a,y),\muu,q(y,a,x))&=&q(q(x,a,y),\muu,q(x,\overline{a},y))\\
                        &=&q(x,p(a,\muu,\overline{a}),y)\\
                        &=&q(x,\muu,y).
\end{eqnarray*}
To prove \ref{Y8}, we start with the important property (\ref{important}) of the underlying mobi algebra and then get:
\begin{eqnarray*}
& &q(x,\muu \cdot p(a,b,c),y)=q(x,\overline{b}\cdot a\oplus b\cdot c,y)\\
&\Rightarrow& q(x,p(0,\muu,p(a,b,c)),y)=q(x,p(\overline{b}\cdot a, \muu,b\cdot c),y)\\
&\Rightarrow& q(x,\muu,q(x,p(a,b,c),y))=q(q(x,\overline{b}\cdot a,y),\muu,q(x,b\cdot c,y))\\
&\Rightarrow& q(x,\muu,q(x,p(a,b,c),y))\\
& &=q(q(x,\overline{b},q(x,a,y)),\muu,q(x,b,q(x,c,y))).
\end{eqnarray*}
It is easy to see that $\ref{Y9}$ is a direct consequence of $\ref{Y5}$, while $\ref{Y10}$ is a consequence of \ref{space_idem} and \ref{space_homo}. 
\end{proof}

In order to have an intuition on the strength of the axioms for a mobi-space, let us take a simple example with the variables $x,y\in\R$ and $t\in[0,1]$. First, let us define a map
\[q({x,t,y})=x\cos\left(t\right)+y\sin\left(t\right)\]
and observe that it satisfies $q(x,0,y)=x$ but not $q(x,1,y)=y$. If we put
\[q({x,t,y})=x\cos\left(t\frac{\pi}{2}\right)+y\sin\left(t\frac{\pi}{2}\right)\]
then we have $q(x,0,y)=x$ and $q(x,1,y)=y$ but axiom \ref{space_idem}, namely $q(x,t,x)=x$, fails for all values other than $x=0$ or $t\in\{0,1\}$.

If we change the map $q$ to be
\begin{equation}
q({x,t,y})=x\cos^2\left(t\frac{\pi}{2}\right)+y\sin^2\left(t\frac{\pi}{2}\right)\label{trying}
\end{equation}
then we get axioms \ref{space_idem} and \ref{space_cancel} but the axiom \ref{space_homo} is not verified. Take for example,  $x=1$, $y=0$, $r=t=\frac{1}{3}$ and $s=1$ and observe that 
\[q(x,r+t(s-r),y)=\cos^2\left(\frac{5\pi}{18}\right)\] while
\[q(q(x,r,y),t,q(x,s,y))=\cos^4\left(\frac{\pi}{6}\right)\] and they are not equal.

One might expect that in order to fix this problem it would be sufficient to find a map $\theta$ such that $$q(x,\theta(r+t(s-r)),y)=q(q(x,\theta(r),y),\theta(t),q(x,\theta(s),y)).$$
However, this is not so simple. Indeed, perhaps a first guess would be to consider the map $\theta(t)=\frac{2}{\pi}\arcsin(t)$. 
With this modification, (\ref{trying}) would become $$q(x,\theta(t),y)=x+(y-x)t^2.$$ 
This new formula, however, still does not satisfy  axiom \ref{space_homo}. 

There is, nevertheless, a general procedure that leads to a mobi space out of the formula $h(x,t,y)=x+(y-x)t^2$, but it involves some extra work. We have to introduce one extra dimension, while solving a certain system of equations (see Theorem \ref{thm:extra dim} below). 

First we solve the system of two equations
\begin{equation*}
\left\{\begin{array}{rcl}
A+(B-A)\,r^2&=&x\\
A+(B-A)\,s^2&=&y
\end{array}\right.
\end{equation*}
which has a unique solution for every $x,y\in\R,r,s\in\R^+$ and $s\neq r$, namely
\begin{equation}\label{alphabeta}
\begin{pmatrix}
A\\ B
\end{pmatrix}
=
\frac{1}{s^2-r^2}
\begin{pmatrix}
s^2 & -r^2\\ -(1-s^2) & 1-r^2
\end{pmatrix}
\begin{pmatrix}
x\\ y
\end{pmatrix}.
\end{equation}

The mobi space on the set $\R\times\R^+$ (over the unit interval) is thus given by the formula
$$
q((x,r),t,(y,s))=(h(A,r+t(s-r),B),r+t(s-r))
$$
with $s\neq r$ and $A$, $B$ obtained from equation $(\ref{alphabeta})$. When $s=r$ we put
$$
q((x,r),t,(y,s))=(x+t(y-x),r).
$$
We end up with the operation $q:(\R\times\R^+)\times[0,1]\times(\R\times\R^+)\to(\R\times\R^+)$ defined by
\begin{equation}\label{t2}
\begin{array}{l}
q\big((x,r),t,(y,s)\big)=\\[0.4mm]
\left\{\begin{array}{lcl}
\left(x+(y-x)\frac{2 r t+(s-r) t^2}{r+s}, r+t(s-r)\right)&, if& s\neq r\\[3mm]
\big(x+t (y-x),r\big)                                 &, if& s= r
\end{array}
\right.,
\end{array}
\end{equation}
turning $\left(\R\times\R^+,q\right)$ into a mobi space over the unit interval.
This procedure provides a way to construct examples of mobi spaces and it is detailed in the next theorem and generalized in Section \ref{sec6}. Further examples are given in the next section. 

\begin{theorem}\label{thm:extra dim}
Consider two real valued functions $f$ and $g$, of one variable. Let $I\subseteq\R$ be an interval of the real numbers such that for any $s,t\in I$, with $t\neq s$, the following condition holds:
\begin{equation}\label{matrix3.2}
f(s)\,g(t)\neq f(t)\,g(s).
\end{equation}
Let $V$ be a real vector space and $K\colon{I\to V}$ be any function.
Then, $(V\times I,q)$ is a mobi space over the mobi algebra $([0,1],p,0,\frac{1}{2},1)$, where  
$$q\colon{(V\times I)\times [0,1]\times (V\times I)\to V\times I}$$ 
is defined by the formula
\begin{equation}
q((x,s),a,(y,t))=(f(p)A+g(p)B-K(p),p),\ \textrm{when}\ t\neq s
\end{equation}
with $p\equiv p(s,a,t)=s+a(t-s)$ and $A,B\in V$ the unique solutions of the system of equations
\begin{equation*}
\left\{\begin{array}{rcl}
f(s)A+g(s)B&=&x+K(s)\\
f(t)A+g(t)B&=&y+K(t)
\end{array}\right.,
\end{equation*}
whereas 
\begin{equation}
q((x,s),a,(y,t))=(x+a(y-x),s), \ \textrm{when}\ t=s.
\end{equation}
\end{theorem}

\noindent The proof will be given in Proposition \ref{Linear Example} from which Theorem \ref{thm:extra dim} is just a particular case. 

For the moment let us look at some examples and then compare the singular case of affine mobi spaces with R-modules.

\section{Examples}
In the following list of examples, unless otherwise stated, the underlying mobi algebra structure $(A, p, 0, \muu, 1)$ is the closed unit interval, i.e. $A=[0,1]$, the three constants are $0, \frac{1}{2}, 1$ and 
$$p(a,b,c)=(1-b)a+b c,\ \textrm{for all}\ a,b,c\in A.$$ We will refer to this algebra as the {\it canonical mobi algebra}. In each case, we present a set $X$ and a ternary operation $q(x,a,y)\in X$, for all \mbox{$x,y\in X$}, and $a \in A$, verifying the axioms of Definition \ref{mobi_space}. We also explain how it was obtained as an instance of some general construction.

\subsection{The canonical mobi space}\label{ex1to4}~\par
\begin{enumerate}
\item \label{ex1} Vector spaces provide examples. For instance:
\[X=\R^n\quad (n\in\mathbb{N})\]
and
\[q(x,a,y)=(1-a)x+a\,y.\]

\item \label{ex4} The well known technique of transporting the structure provides us with other ways of presenting the canonical structure. For every bijective map $F\colon{X\to X}$, with $X\subseteq \R^n$, we get a mobi space $(X,q)$ with
\begin{eqnarray*}
q(x,a,y)=F^{-1}((1-a)F(x)+a\,F(y)).
\end{eqnarray*}
For instance, in the case of dimension one:
\begin{enumerate}
\item \label{ex2} If $F(x)=\log x$ and $X=\mathbb{R}^{+}$ then we get the mobi space $(X,q)$, with
\[q(x,a,y)=x^{1-a} y^a.\]

\item \label{ex3} If $F(x)=\frac{1}{x}$, then $(\R^+,q)$ is a mobi space with
\[q(x,a,y)=\dfrac{x y}{a x+(1-a) y}.\]
\end{enumerate}
\end{enumerate}

\subsection{Examples obtained directly from Theorem \ref{thm:extra dim} }\label{ex5to9}~\par
To apply Theorem \ref{thm:extra dim}, we need three functions $f$, $g$ and $K$ such that $f(s)g(t)\neq f(t)g(s)$ for all $s,t\in I, s\neq t$. If $g$ is a non-zero constant function, this condition just imply the injectivity of $f$. Let us begin with $K=0$, $g=1$ and a function $f$ injective in $I$.

\begin{enumerate}

\item \label{ex6} With $f\colon{\R^+\to\R};\,x\mapsto x^2$, we obtain 
\[X=\mathbb{R}\times\mathbb{R}^+\]
with the formula

\[q((x,s),a,(y,t))=\left(x+(y-x)\frac{2\,s\,a+(t-s) a^2}{t+s},s+a(t-s)\right).\]
This is, in fact, the example displayed in equation (\ref{t2}).
Note that, since $r,s\in\R^+$, the second branch in (\ref{t2}) is not necessary.

\item \label{ex7} With $f\colon{\R^+\to\R};\,x\mapsto \frac{1}{x}$, we get the set
\[X=\mathbb{R}\times\mathbb{R}^+\]
with the formula

\[q((x,s),a,(y,t))=\left(x+(y-x)\frac{a\,t}{(1-a)s+a\,t},s+a(t-s)\right).\]

\item \label{ex8} With $f\colon{\R\to\R};\,x\mapsto x^3$, we can consider the set
\[X=\mathbb{R}^2\]
and the formula
$$
\begin{array}{l}
q((x,s),a,(y,t))=\\[3mm]
\left(x+(y-x)\dfrac{3\,s^2\,a+3\,s(t-s)a^2+(t-s)^2\,a^3}{s^2+s\,t+t^2},s+a(t-s)\right),
\end{array}
$$
if $(s,t)\neq (0,0)$, and
\[q((x,0),a,(y,0))=\left(x+a\,(y-x),0\right).\]

\item \label{ex9} In general, applying Theorem \ref{thm:extra dim} with $g=1$, $K=0$ and $f$ an injective real function of one variable, we get a mobi space in any set $X\subseteq\mathbb{R}^2$ for which the formula 
\begin{equation}\label{exgeneral}
\begin{array}{l}
q((x,s),a,(y,t))=\\[3mm]
\left(x+(y-x)\dfrac{f(s+(t-s) a)-f(s)}{f(t)-f(s)},s+a(t-s)\right),
\end{array}
\end{equation}
if $s\neq t$, and
\[q((x,s),a,(y,s))=\left(x+a\,(y-x),s\right),\]
defines a map $q\colon{X\times[0,1]\times X\to X}$.   
\end{enumerate}
At first glance, we could be skeptical that this operation $q$ verify even the simplest properties of a mobi space, like \ref{Y1}, for an arbitrary injective function $f$, but it does. Indeed, for $t\neq s$, (\ref{exgeneral}) imply  
$$\begin{array}{l}
q((x,s),1-a,(y,t))\\
=\left(x+(y-x)\dfrac{f(t+a(s-t))-f(s)}{f(t)-f(s)},t+a(s-t)\right)\\
=\left(x+(y-x)\dfrac{f(t+a(s-t))-f(t)+f(t)-f(s)}{f(t)-f(s)},t+a(s-t)\right)\\
=\left(y+(x-y)\dfrac{f(t+a(s-t))-f(t)}{f(s)-f(t)},t+a(s-t)\right)\\
=q((y,t),a,(x,s)).
\end{array}$$

We will now see some examples obtained from physics.

\subsection{Examples with physical interpretation}\label{physics}~\par
The following examples, from classical mechanics, can be viewed as an application of Theorem \ref{thm:extra dim} with specific expressions of $f$, $g$ and $K$.

\begin{enumerate}

\item \label{ex11}\label{projectiles} Consider a constant acceleration motion, with $x\in\R^n$, and the following position equation
$$ x(t)=x_0+v_0\,t-k\,t^2.$$
We can think, for instance, of a projectile motion in the plane $\R^2$ where $k$ would be $(0,\frac{g}{2})$ with $g$ being the gravitational acceleration near the Earth's surface. The constants $x_0$ and $v_0$ correspond to the usual initial conditions $x(0)=x_0$ and $x'(0)=v_0$. Trying to construct a mobi space out of this context, without extra dimension, we could think of imposing boundary conditions like $x(0)=x_0$ and $x(1)=x_1$, leading to:
$$x(t)=x_0+(x_1-x_0)\,t+k\,t\,(1-t).$$
But then the operation $q$ defined as $q(x_0,t,x_1)=x(t)$ is not a mobi operation: in particular, idempotency $q(x,t,x)=x$ is not verified because a body could go up vertically and then down back to the same place; axiom \ref{space_homo} is not verified either. The way to go is to let the variable $t$ flow freely in an extra dimension with boundary conditions like $x(t_1)=x_1$ and $x(t_2)=x_2$. These conditions lead to:
$$\qquad x(t_1+a(t_2-t_1))=x_1+a(x_2-x_1)+k\,a\, (1-a) (t_2-t_1)^2. $$
In the scope of Theorem \ref{thm:extra dim}, we could say that $f(t)=t$, $g(t)=1$ and $K(t)=k\,t^2$.
For any $k\in\R^n$, we have then a mobi space $(X,q)$ over the canonical mobi algebra by taking the set
\[X=\mathbb{R}^{n+1}\]
with the formula
\begin{equation}\label{eqprojectiles}
\begin{array}{l}
q((x,s),a,(y,t))=\\[0.3mm]
\left(x+a (y-x)+k\, a\, (1-a)(t-s)^2, s+a\,(t-s)\right).
\end{array}
\end{equation}

Remark: 
This example could be generalized to Special Relativity \cite{relativistic_projectiles}. However, the operation $q$ is then a partial operation because, in Minkowski space-time, not every two points can be reached from one another if one point is not inside the {\it light cone} of the other. 
 
\item \label{ex:fxdot}
The solutions for the one-dimension motion of the well-known damped harmonic oscillator are of the form
$$ A f(t)+B g(t)-K(t),$$
where A and B are real parameters. If the oscillator is not driven, K(t)=0. Depending on the circumstances, we can have:
$$\qquad
\begin{array}{ll}
\textrm{overdamping}: f(t)=e^{\alpha\,t},\ g(t)=e^{\beta\,t},\ \alpha\neq\beta\\
\textrm{critical damping}: f(t)=e^{\alpha\,t},\ g(t)=t\,e^{\alpha\,t}\\
\textrm{underdamping}:f(t)=e^{\alpha\,t}\,\sin(\beta\,t),\ g(t)=e^{\alpha\,t}\,\cos(\beta\,t),\ \beta\neq 0,
\end{array}
$$
where $\alpha,\beta\in\R$ depend on the oscillatory system. For the first two cases, the condition (\ref{matrix3.2}) is verified for all  $s,t\in \R$, with $t\neq s$ and therefore we can apply \mbox{Theorem \ref{thm:extra dim}}. 
\begin{enumerate}
\item\label{criticaldamping} In the critical damping case and for any $\alpha\in\R$, we obtain the following mobi space $(\mathbb{R}^2,q)$ over the canonical mobi algebra
with the formula
$$
\begin{array}{l}
q((x,s),a,(y,t))=\\[3mm]
\left((1-a)\,x\,e^{\alpha a (t-s)}+a\, y\,e^{\alpha (1-a) (s-t)}, s+a(t-s)\right).
\end{array}$$
\item\label{overdamping}
In the overdamping case and for any $\alpha,\beta\in\R, \alpha\neq\beta$, we obtain the mobi space $(\mathbb{R}^2,q)$ over the canonical mobi algebra where q is defined, for $t\neq s$, by
$$
\begin{array}{l}
\qquad q((x,s),a,(y,t))=
\Big(\dfrac{e^{\alpha(1-a)(s-t)}-e^{\beta(1-a)(s-t)}}{e^{\alpha s+\beta t}-e^{\alpha t+\beta s}}\,e^{(\alpha+\beta)t}\,x\\[3mm]
\qquad+\dfrac{e^{\beta a (t-s)}-e^{\alpha a (t-s)}}{e^{\alpha s+\beta t}-e^{\alpha t+\beta s}}\,e^{(\alpha+\beta)s}\,y
, s+a(t-s)\Big),
\end{array}$$
and, for $t=s$, by $q((x,s),a,(y,s))=(x+a(y-x),s)$. 
\item\label{underdamping} For the case of underdamping, we can still apply Theorem~\ref{thm:extra dim} if we restrict the possible values of $s$ and $t$ to, for instance, $I=[0,\pi[$. Here, we just mention that the case where $f(t)=\sin(\beta t)$ and $g(t)=\cos(\beta t)$ is analysed in Section \ref{sec.Slerp}.
\end{enumerate}
It is interesting to note that example (\ref{criticaldamping}) can be obtained from example (\ref{overdamping}) in the limit situation $\beta\to\alpha$.

\subsection{An example over a different mobi algebra}\label{ex12}\label{ex_lozenge}~\par
So far we have considered examples of mobi spaces over the unit interval. Here is an example with a different mobi algebra.
For the mobi algebra $(A,p,0,\muu,1)$ let us use
 $$A=\left\{(t_1,t_2)\in\mathbb{R}^2\colon \vert t_2\vert \leq t_1 \leq 1-\vert t_2\vert\right\}$$
$$\muu=\left(\frac{1}{2},0\right)\,;\,1=(1,0)\,;\,0=(0,0)$$
\begin{eqnarray*}
p(a,b,c)=(a_1-b_1 a_1-b_2 a_2+b_1 c_1+b_2 c_2,\\
          a_2-b_1 a_2-b_2 a_1+b_1 c_2+b_2 c_1).
\end{eqnarray*}
And for the mobi space $(X,q)$:
$X=[0,1]$ and, with $h=\pm 1$, $$q(x,(t,s),y)= (1-t-h\,s) x+(t+h\,s)y.$$

\end{enumerate}

\subsection{Affineness of the examples}\label{affineness}~\par
We end this section with 
some comments on whether the examples presented verify the affine condition (\ref{affine}) or not.
Examples like those corresponding to example \ref{ex5to9}(\ref{ex9}) are, in general, not affine in the sense that they don't verify 
$$
q\left(q[(x_1,s_1),a,(y_1,t_1)],\frac{1}{2},q[(x_2,s_2),a,(y_2,t_2)]\right)$$
$$=q\left(q[(x_1,s_1),\frac{1}{2},(x_2,s_2)],a,q[(y_1,t_1),\frac{1}{2},(y_2,t_2)]\right).
$$
Indeed, in example \ref{ex5to9}(\ref{ex6}) for instance, we have that
$$q\left(q[(0,0),\frac{1}{3},(0,1)],\frac{1}{2},q[(1,1),\frac{1}{3},(0,0)]\right)=\left(\frac{5}{27},\frac{1}{2}\right)$$
but
$$q\left(q[(0,0),\frac{1}{2},(1,1)],\frac{1}{3},q[(0,1),\frac{1}{2},(0,0)]\right)=\left(\frac{1}{6},\frac{1}{2}\right).$$
Similarly, in example \ref{ex5to9}(\ref{ex8}), we have for instance:
$$q\left(q[(0,0),\frac{1}{3},(0,1)],\frac{1}{2},q[(1,1),\frac{1}{3},(0,0)]\right)=\left(\frac{19}{189},\frac{1}{2}\right)$$
while
$$q\left(q[(0,0),\frac{1}{2},(1,1)],\frac{1}{3},q[(0,1),\frac{1}{2},(0,0)]\right)=\left(\frac{1}{12},\frac{1}{2}\right).$$
O course, the canonical mobi spaces \ref{ex1to4} are affine. Examples \ref{ex5to9}(\ref{ex7}), \ref{physics}(\ref{projectiles}) and \ref{physics}(\ref{criticaldamping}) correspond also to affine mobi spaces while \ref{physics}(\ref{overdamping}) does not. Indeed, for \ref{physics}(\ref{overdamping}), we have for instance that
$$\begin{array}{l}
q\left(q[(0,0),\frac{1}{3},(0,1)],\frac{1}{6},q[(1,1),\frac{1}{3},(0,0)]\right)\\[5mm]
=\left(\dfrac{(e^{\alpha/18}-e^{\beta/18})(e^{\alpha/3}+e^{\beta/3})}{e^\alpha-e^\beta},\frac{7}{18}\right)
\end{array}$$
but
$$\begin{array}{l}
q\left(q[(0,0),\frac{1}{6},(1,1)],\frac{1}{3},q[(0,1),\frac{1}{6},(0,0)]\right)\\[5mm]
=\left(e^{2(\alpha+\beta)/3}\dfrac{(e^{-4\alpha/9}-e^{-4\beta/9})(e^{\beta/6}-e^{\alpha/6})}{(e^{2\alpha/3}-e^{2\beta/3})(e^\alpha-e^\beta)},\frac{7}{18}\right).
\end{array}$$
The two results are different if $\alpha\neq \beta$. However, in the limit situation when $\beta\to\alpha$, the critical case is recovered and the two results are naturally equal. The example \ref{physics}(\ref{underdamping}) is not affine either.

In the next section, we compare affine mobi spaces \cite{preprint} with modules over a ring with one-half.

\section{Comparison  with R-modules}\label{sec:modules}

Consider a unitary ring $(R,+,\cdot,0,1)$. It has been proven \cite{mobi} that if $R$ contains the inverse of $1+1$, then it is a mobi algebra and if a mobi algebra $(A,p,0,\muu,1)$ contains the inverse of $\muu$, in the sense of (\ref{def_product}), then it is a unitary ring. In this section, we will compare a module over a ring $R$ with a mobi space over a mobi algebra A. First, let us just recall that a module over a ring $R$ is a system $(M,+,e,\varphi)$, where $\varphi:R\to \End(M)$ is a map from $R$ to the usual ring of endomorphisms, such that $(M,+,e)$ is an abelian group and $\varphi$ is a ring homomorphism.

The following theorem shows how to construct a mobi space from a module over a ring containing the inverse of $2$.
\begin{theorem}\label{module2mobi}
Consider a module $(X,+,e,\varphi)$ over a unitary ring $(A,+,\cdot,0,1)$. If $A$ contains $(1+1)^{-1}=\muu$ then $(X,q)$ is an affine mobi space over the mobi algebra $(A,p,0,\muu,1)$, with
\begin{eqnarray}
p(a,b,c)&=&a+b c-b a\\
q(x,a,y)&=&\varphi_{1-a}(x)+\varphi_a(y).\label{q}
\end{eqnarray}
\end{theorem}
\begin{proof}
$(A,p,0,\muu,1)$ is a mobi algebra by Theorem 7.2 of \cite{mobi}. We show here that the axioms of Definition \ref{mobi_space}, as well as (\ref{affine}), are verified. The first three axioms are easily proved:
\begin{eqnarray*}
q(x,0,y)&=&\varphi_1(x)+\varphi_0(y)=x+ e=x\\
q(x,1,y)&=&\varphi_0(x)+\varphi_1(y)=e+ y=y\\
q(x,a,x)&=&\varphi_{1-a}(x)+\varphi_a(x)=\varphi_{1-a+a}(x)=\varphi_1(x)=x.
\end{eqnarray*}
Axiom \ref{space_cancel} is due to the fact that $\muu+\muu=1$ and consequently
\begin{eqnarray*}
\varphi_\muu(y_1)=\varphi_\muu(y_2)&\Rightarrow& \varphi_\muu(y_1)+\varphi_\muu(y_1)=\varphi_\muu(y_2)+\varphi_\muu(y_2)\\                                   &\Rightarrow& \varphi_1(y_1)=\varphi_1(y_2)\Rightarrow y_1=y_2.
\end{eqnarray*}
Next, we give a proof of Axiom \ref{space_homo}. It is relevant to notice that, besides other evident properties of the module $X$, the associativity of~$+$ plays an important part in the proof:
\begin{eqnarray*}
&&q(q(x,a,y),b,q(x,c,y))\\
&=&\varphi_{1-b}(\varphi_{1-a}(x)+\varphi_a(y))+\varphi_b(\varphi_{1-c}(x)+\varphi_c(y))\\
&=&\varphi_{1-b}(\varphi_{1-a}(x))+\varphi_{1-b}(\varphi_a(y))+\varphi_b(\varphi_{1-c}(x))+\varphi_b(\varphi_c(y))\\
&=&\varphi_{(1-b)(1-a)}(x)+\varphi_{b(1-c)}(x)+\varphi_{(1-b)a}(y)+\varphi_{bc}(y)\\
&=&\varphi_{1-a+ba-bc}(x)+\varphi_{a-ba+bc}(y)\\
&=&\varphi_{1-p(a,b,c)}(x)+\varphi_{p(a,b,c)}(y)\\
&=&q(x,p(a,b,c),y).
\end{eqnarray*}
It remains to prove (\ref{affine}):
\begin{eqnarray*}
&&q(q(x_1,a,y_1),\muu,q(x_2,a,y_2))\\
&=&\varphi_{\muu}(\varphi_{1-a}(x_1)+\varphi_a(y_1))+\varphi_\muu(\varphi_{1-a}(x_2)+\varphi_a(y_2))\\
&=&\varphi_{\muu}(\varphi_{1-a}(x_1)+\varphi_a(y_1)+\varphi_{(1-a)}(x_2)+\varphi_a(y_2))\\
&=&\varphi_{\muu}(\varphi_{1-a}(x_1+ x_2)+\varphi_a(y_1+ y_2))\\
&=&\varphi_{(1-a)\muu}(x_1+ x_2)+\varphi_{a\muu}(y_1+ y_2)\\
&=&\varphi_{(1-a)}(\varphi_\muu(x_1)+\varphi_\muu(x_2))+\varphi_a(\varphi_{\muu}(y_1)+\varphi_{\muu}(y_2))\\
&=&\varphi_{(1-a)}(q(x_1,\muu,x_2))+\varphi_{a}(q(y_1,\muu,y_2))\\
&=&q(q(x_1,\muu,x_2)),a,q(y_1,\muu,y_2)).
\end{eqnarray*}
\end{proof}

\begin{theorem}\label{mobi2module}
Consider an affine mobi space $(X,q)$, with a fixed chosen element $e\in X$, over a mobi algebra $(A,p,0,\muu,1)$. If $A$ contains $2$ such that $p(0,\muu,2)=1$ then $(X,+,e,\varphi)$ is a module over the unitary ring $(A,+,\cdot,0,1)$, with
\begin{eqnarray}
a+b&=&p(0,2,p(a,\muu,b))\\
a\cdot b&=&p(0,a,b)\\
\varphi_a(x)&=&q(e,a,x)\\
x+ y&=&q(e,2,q(x,\muu,y))=\varphi_2(q(x,\muu,y)).
\end{eqnarray}
\end{theorem}
\begin{proof}
$(A,+,\cdot,0,1)$ is a unitary ring by Theorem 7.1 of \cite{mobi}. We prove here that $(X,+,e,\varphi)$ is a module over $A$. First, we observe that, using in particular \ref{Y3} of Proposition \ref{properties_space}, we have:
\begin{eqnarray*}
q(e,\muu,x+y)&=& q(e,\muu,q(e,2,q(x,\muu,y)))\\
&=& q(e,\muu\cdot 2,q(x,\muu,y))\\
                            &=& q(e,1,q(x,\muu,y))\\
														&=& q(x,\muu,y).
\end{eqnarray*}
Then, the property (\ref{affine}) of an affine mobi space is essential to prove the associativity of the operation $+$ of the module:
\begin{eqnarray*}
q(e,\muu,q(e,\muu,(x+ y)+ z))&=&q(q(e,\muu,e),\muu,q(x+ y,\muu,z))\\
&=& q(q(e,\muu,x+ y),\muu,q(e,\muu,z))\\
&=& q(q(x,\muu,y),\muu,q(e,\muu,z))\\
&=& q(q(x,\muu,e),\muu,q(y,\muu,z))\\
&=& q(q(e,\muu,x),\muu,q(e,\muu,y+ z))\\
&=& q(q(e,\muu,e),\muu,q(x,\muu,y+ z))\\
&=& q(e,\muu,q(e,\muu,x+(y+ z)))\\
\end{eqnarray*}
Which, by \ref{space_cancel}, implies that $(x+ y)+ z=x+(y+ z)$. Commutativity of $+$ and the identity nature of $e$ are easily proved:
\begin{eqnarray*}
q(e,\muu,e+ x)&=&q(e,\muu,x)\Rightarrow e+ x=x\\
q(e,\muu,x+ y)&=&q(x,\muu,y)=q(y,\muu,x)\\
                    &=&q(e,\muu,y+ x)\Rightarrow x+ y=y+ x.
\end{eqnarray*}
Cancellation is achieved with $-x=q(e,p(1,2,0),x)$. Indeed:
\begin{eqnarray*}
q(e,\muu,q(e,p(1,2,0),x)+ x)&=&q(q(e,p(1,2,0),x),\muu,x)\\
                                  &=&q(q(e,p(1,2,0),x),\muu,q(e,1,x))\\
                                  &=&q(e,p(p(1,2,0),\muu,1),x)\\
                                  &=&q(e,p(1,p(2,\muu,0),0),x)\\
                                  &=&q(e,p(1,1,0),x)\\
                                  &=&q(e,0,x)=e=q(e,\muu,e)\\
\end{eqnarray*}
To prove that $\varphi_a(x+y)=\varphi_a(x)+\varphi_a(y)$, we will again need (\ref{affine}):
\begin{eqnarray*}
q(e,\muu,\varphi_a(x+y))&=& q(e,\muu,q(e,a,x+ y))\\    
														   &=& q(q(e,a,e),\muu,q(e,a,x+ y))\\   
														   &=& q(q(e,\muu,e),a,q(e,\muu,x+ y))\\   
														   &=& q(e,a,q(x,\muu,y))\\ 
														   &=& q(q(e,a,x),\muu,q(e,a,y))\\  
														   &=& q(\varphi_a(x),\muu,\varphi_a(y))\\ 
														   &=& q(e,\muu,\varphi_a(x)+\varphi_a(y)).
\end{eqnarray*}
To prove that $\varphi_{a+b}(x)=\varphi_a(x)+\varphi_b(x)$, let us first recall that, in a mobi algebra with 2 and $a+b=p(0,2,p(a,\muu,b))$, we have the following property:
$$p(0,\muu,a+b)=p(a,\muu,b).$$
We then have
\begin{eqnarray*}
q(e,\muu,\varphi_{a+b}(x))&=& q(e,\muu,q(e,a+b,x))\\    
													&=& q(q(e,0,x),\muu,q(e,a+b,x))\\    
													&=& q(e,p(0,\muu,a+b),x)\\    
													&=& q(e,p(a,\muu,b),x)\\    
													&=& q(q(e,a,x),\muu,q(e,b,x))\\    
													&=& q(\varphi_a(x),\muu,\varphi_b(x))\\
													&=& q(e,\muu,\varphi_a(x)+\varphi_b(x)).
\end{eqnarray*}
The last two properties are easily proved:
$$\varphi_{a\cdot b}(x)=q(e,a\cdot b,x)=q(e,a,q(e,b,x))=\varphi_a(\varphi_b(x))$$
$$\varphi_1(x)=q(e,1,x)=x.$$
\end{proof}

\begin{proposition}\label{module2module}
Consider a R-module $(X,+,e,\varphi)$ within the conditions of Theorem \ref{module2mobi} and the corresponding mobi space $(X,q)$. Then the R-module obtained from $(X,q)$ by Theorem \ref{mobi2module} is the same as $(X,+,e,\varphi)$.
\end{proposition}
\begin{proof}
From $(X,q)$, we define
$$x+'y=q(e,2,q(x,\muu,y))\ \textrm{and}\ \varphi_a'(x)=q(e,a,x)$$
and obtain the following equalities:
\begin{eqnarray*}
x+'y&=&e+\varphi_2(q(x,\muu,y))\\
    &=&\varphi_2(\varphi_\muu(x)+\varphi_\muu(y))\\
		&=&\varphi_{2\cdot\muu}(x)+\varphi_{2\cdot\muu}(y)\\
		&=&x+y
\end{eqnarray*}
$$\varphi'_a(x)=\varphi_{1-a}(e)+\varphi_a(x)=e+\varphi_a(x)=\varphi_a(x).$$
\end{proof}

\begin{proposition}\label{mobi2mobi}
Consider an affine mobi space $(X,q)$ within the conditions of Theorem \ref{mobi2module} and the corresponding module $(X,+,e,\varphi)$. Then the affine mobi space obtained from $(X,+,e,\varphi)$ by Theorem \ref{module2mobi} is the same as $(X,q)$.
\end{proposition}
\begin{proof}
From $(X,+,e,\varphi)$, we define
$$ q'(x,a,y)=\varphi_{1-a}(x)+\varphi_a(y)$$
and obtain the following equalities:
\begin{eqnarray*}
q'(x,a,y)&=& q(e,\overline{a},x)+q(e,a,y)\\
         &=& q(x,a,e)+q(e,a,y)\\
         &=& q(e,2,q(q(x,a,e),\muu,q(e,a,y))).
\end{eqnarray*}
Now, because we are considering that $(X,q)$ is affine, we get:
\begin{eqnarray*}
q'(x,a,y)&=& q(q(e,a,e),2,q(q(x,\muu,e),a,q(e,\muu,y)))\\
         &=& q(q(e,2,q(e,\muu,x)),a,q(e,2,q(e,\muu,y)))\\
         &=& q(q(e,2\cdot \muu,x),a,q(e,2\cdot\muu,y))\\
				 &=& q(x,a,y).
\end{eqnarray*}
\end{proof}

We have completely characterized affine mobi spaces in terms of modules over a unitary ring in which 2 is invertible. In a sequel to this paper we will investigate how to characterize mobi spaces in terms of  homomorphisms between mobi algebras. 

We finish this section by taking a closer look to Example \ref{physics}(\ref{projectiles}) 
and a related module. 
This is an example of an affine mobi space and can be extended to the case where the underlying mobi algebra is $(\mathbb{R},p,0,\frac{1}{2},1)$. Then by Theorem \ref{mobi2module} we get a module over $(\R,+,\cdot,0,1)$ given by $(\R^{n+1},+,0,\varphi)$ with $x,y,k\in\R^n$, $s,t\in\R$ and
\begin{eqnarray*}
(x,s)+(y,t)&=&(x+y-2 k s t,s+t)\\
\varphi_a(x,s)&=&(a x+k a (1-a) s^2,a s).
\end{eqnarray*}
By Theorem \ref{module2mobi}, we can construct a mobi space from this module and verify that it is the same as (\ref{eqprojectiles}). Of course, in this example, the module is a vector field and there is a homomorphism, namely
$$f(x,s)=(x+k (s^2-s),s)$$
from $(\R^{n+1},+,0,\varphi)$ to the usual vector field in $\R^{n+1}$.

In the following section we will thoroughly analyse  a procedure to construct examples of mobi spaces which in general are not affine mobi spaces. In a sequel to this work we will investigate the case of spaces with geodesics and how to construct mobi spaces out of them.

\section{Non-affine mobi spaces}\label{sec6}

We present here a general result from which Theorem \ref{thm:extra dim} can be deduced. In general, the examples that are obtained in this way are not affine.

\begin{proposition}\label{General-Example}

Let $(X,q_X)$ and $(Y,q_Y)$ be two mobi spaces over a mobi algebra $(A,p)$. Suppose we have two sets $U$, $V$ and a function $h\colon{U\times Y\times V\to X}$ such that the system 
\begin{equation}\label{system_h}
\left\{\begin{array}{rcl}
h(\alpha,y_1,\beta)&=&x_1\\
h(\alpha,y_2,\beta)&=&x_2
\end{array}\right.
\end{equation}
has a unique solution for every $x_1, x_2 \in X$ and any $y_1, y_2 \in Y$ with $y_1\neq y_2$, namely 
\begin{equation}\label{alpha-beta}
\left\{\begin{array}{rcl}
\alpha&=&\alpha(x_1,y_1,x_2,y_2)\\
\beta&=&\beta(x_1,y_1,x_2,y_2)
\end{array}\right..
\end{equation}
Then, $(X\times Y,q)$ is a mobi space over the mobi algebra $(A,p)$ where
$$ q\colon{(X\times Y)\times A\times (X\times Y)\to (X\times Y)}$$
is defined using the map $\chi$, via (\ref{alpha-beta}), \begin{equation}\label{chi}
\chi(x_1,y_1,a,x_2,y_2)=
\left\{\begin{array}{rcl}
h[\alpha,q_Y(y_1,a,y_2),\beta] &\textrm{ if } &y_1\neq y_2\\ 
q_X(x_1,a,x_2)&\textrm{ if } &y_1= y_2
\end{array}\right..
\end{equation}
as
$$ q\left((x_1,y_1),a,(x_2,y_2)\right)=\left(\chi(x_1,y_1,a,x_2,y_2),q_Y(y_1,a,y_2)\right).$$
\end{proposition}
\begin{proof}
The axioms \ref{space_0}, \ref{space_1} and \ref{space_idem} are direct consequences of (\ref{system_h}) and the fact that $q_Y$ and $q_X$ are operations of mobi spaces. To prove \ref{space_cancel}, we first observe that 
\begin{equation}\label{Axiom4}
q\left[(x_1,y_1),\muu,(x_2,y_2)\right]=q\left[(x_1,y_1),\muu,(x'_2,y'_2)\right]
\end{equation}
implies $q_Y(y_1,\muu,y_2)=q_Y(y_1,\muu,y'_2)$ and hence $y'_2=y_2$. If $y_2=y_1$, we also get $q_X(x_1,\muu,x_2)=q_X(x_1,\muu,x'_2)$ and consequently $x'_2=x_2$. When $y_2\neq y_1$, $y'_2=y_2$ and (\ref{Axiom4}) imply
$$
h[\alpha,q_Y(y_1,\muu,y_2),\beta]
= h[\alpha',q_Y(y_1,\muu,y_2),\beta']\equiv x_3,$$
where $\alpha'=\alpha(x_1,y_1,x'_2,y_2)$ and $\beta'=\beta(x_1,y_1,x'_2,y_2)$.
Now, because $y_2\neq y_1\Rightarrow q_Y(y_1,\muu,y_2)\neq y_1$, the system
$$\left\{\begin{array}{rcl}
h(\alpha,y_1,\beta)&=&x_1\\
h(\alpha,q_Y(y_1,\muu,y_2),\beta)&=&x_3
\end{array}\right.$$
has a unique solution, we then conclude that $\alpha=\alpha'$ and $\beta=\beta'$ and consequently that 
$$x'_2=h(\alpha',y_2,\beta')=h(\alpha,y_2,\beta)=x_2.$$ 
Let us now prove \ref{space_homo}.
We have to prove that $Q_1=Q_2$ where:
$$Q_1\equiv q\left[(x_1,y_1),p(a,b,c),(x_2,y_2)\right]$$
$$Q_2\equiv q\Big(q\left[(x_1,y_1),a,(x_2,y_2)\right],b,q\left[(x_1,y_1),c,(x_2,y_2)\right]\Big) .$$
To simplify the presentation of the proof, the following notations are used:
$$y_a=q_Y(y_1,a,y_2),\quad y_c=q_Y(y_1,c,y_2),$$
$$\chi_a=h(\alpha,y_a,\beta),\quad\chi_c=h(\alpha,y_c,\beta).$$
\begin{itemize}
\item Considering $y_1\neq y_2$ and $y_a \neq y_c$, we have
$$ \begin{array}{rcl}
Q_1&=&\left(h[\alpha,q_Y(y_1,p(a,b,c),y_2),\beta],q_Y[y_1,p(a,b,c),y_2]\right)\\[10pt]
   &=&\left(h[\alpha,q_Y(y_a,b,y_c),\beta],q_Y[y_a,b,y_c]\right)
\end{array}$$
and
$$\begin{array}{rcl}
Q_2&=&q\Big[(\chi_a,y_a),b,(\chi_c,y_c)\Big]\\[10pt]
   &=&(h[\tilde{\alpha},q_Y(y_a,b,y_c),\tilde{\beta}],q_Y[y_a,b,y_c])
\end{array},$$
where $\tilde{\alpha}$ and $\tilde{\beta}$ are the unique solutions of the system
$$\left\{\begin{array}{rcl}
h(\tilde{\alpha},y_a,\tilde{\beta})&=&\chi_a\\
h(\tilde{\alpha},y_c,\tilde{\beta})&=&\chi_c
\end{array}\right.$$
which imply that $\tilde{\alpha}=\alpha$ and $\tilde{\beta}=\beta$, by definition of $\chi_a$ and $\chi_c$ and because $y_a\neq y_c$, therefore $Q_1=Q_2$.
\item Considering $y_1= y_2$, and hence $y_a = y_c=y_1$, we have
$$\begin{array}{rcl}
Q_1&=&\Big(q_X[x_1,p(a,b,c),x_2],q_Y[y_1,p(a,b,c),y_2]\Big)\\[10pt]
   &=&\Big(q_X[q_X(x_1,a,x_2),b,q_X(x_1,c,x_2)],y_1\Big)
\end{array}$$
and
$$\begin{array}{rcl}
Q_2&=&q\Big((q_X[x_1,a,x_2],y_a),b,(q_X[x_1,c,x_2],y_c)\Big)\\[10pt]
   &=&\Big(q_X[q_X(x_1,a,x_2),b,q_X(x_1,c,x_2)],q_Y(y_a,b,y_c)\Big)\\[10pt]
   &=&\Big(q_X[q_X(x_1,a,x_2),b,q_X(x_1,c,x_2)],y_1\Big)
\end{array}$$
implying that $Q_1=Q_2$.
\item Considering $y_1\neq y_2$ and $y_a = y_c$, hence $\chi_a=\chi_c$, we have
$$\begin{array}{rcl}
Q_1&=&\left(h[\alpha,q_Y(y_a,b,y_c),\beta],q_Y[y_a,b,y_c]\right)\\
   &=&(\chi_a,y_a)
\end{array}$$
and
$$\begin{array}{rcl}
Q_2&=&q\Big[(\chi_a,y_a),b,(\chi_c,y_c)\Big]\\[10pt]
   &=&\Big(q_X[\chi_a,b,\chi_c],q_Y[y_a,b,y_c])\Big)\\[10pt]
   &=&(\chi_a,y_a)\\[10pt]
	 &=&Q_1.
\end{array}$$
\end{itemize}
\end{proof}
As an example, consider $U=X=\R$, $V=\R^+$, $Y=\R_0^+$, $(A,p)$ the canonical mobi algebra and $h(\alpha, y,\beta)=\alpha\, \beta^y.$ Then, for $y_1\neq y_2$,
$$h[\alpha(x_1,y_1,x_2,y_2),t,\beta(x_1,y_1,x_2,y_2)]=x_1^{\frac{t-y_2}{y_1-y_2}}\,x_2^{\frac{y_1-t}{y_1-y_2}},$$
and if $t$ is $q_Y(y_1,a,y_2)=y_1+a\,(y_2-y_1)$, we get:
$$q\left[(x_1,y_1),a,(x_2,y_2)\right]=(x_1^{1-a} x_2^a,y_1+a\,(y_2-y_1)).$$
This expression is well-defined even for $y_2=y_1$. This leaves no option for $q_X$ if we want a continuous operation, as the only possibility is $q_X(x_1,a,x_2)=x_1^{1-a} x_2^a$. But any other mobi operation is allowed when $y_2=y_1$ and we can write:
$$q\left[(x_1,y_1),a,(x_2,y_1)\right]=(q_X(x_1,a,x_2),y_1).$$
This example compares with Example \ref{ex1to4}(\ref{ex2}).
Note that in Example \ref{ex5to9}(\ref{ex8}), the branch corresponding to $(s,t)=(0,0)$ cannot be obtained by continuity due to the fact that the limit $(s,t)\to(0,0)$ does not exist. However, the {\it canonical} expression at $(0,0)$ is the choice which corresponds to approaching the origin through the path $t=s$. 

A useful particular case is when $X$ is a vector space and $h(\alpha,t,\beta)=\alpha f(t) + \beta g(t)$ for some scalar maps $f$ and $g$. Note that Theorem~\ref{thm:extra dim} is a reformulation of the following proposition.

\begin{proposition}\label{Linear Example}

Let $(X,q_X)$ and $(Y,q_Y)$ be two mobi spaces over a mobi algebra $(A,p)$. Suppose moreover that $X$ is  a vector space over a scalar field $F$ and let $f:Y\to F$ and $g:Y\to F$ be two functions such that, for any $y_1, y_2 \in Y$ with $y_1\neq y_2$, the following inequality holds
\begin{equation}\label{system}
f(y_1)\,g(y_2)\neq g(y_1)\, f(y_2) .
\end{equation}
Furthermore, we consider a function $K:Y\to X$.
Then $(X\times Y,q)$ is a mobi space over $(A,p)$ considering that
$$ q:(X\times Y)\times A\times (X\times Y)\to (X\times Y)$$
is defined as
$$q\left[(x_1,y_1),a,(x_2,y_2)\right]=\Big(\chi_a(x_1,y_1,x_2,y_2),q_Y(y_1,a,y_2)\Big)$$
with
$$ \begin{array}{l}
\chi_a\left(x_1,y_1,x_2,y_2\right)\\[20pt]
=\dfrac{g(y_2)\,\left(x_1+K(y_1)\right)-g(y_1)\,\left(x_2+K(y_2)\right)}{f(y_1)\,g(y_2)-f(y_2)\,g(y_1)}\, f[q_Y(y_1,a,y_2)]
\\[20pt]
-\dfrac{f(y_2)\,\left(x_1+K(y_1)\right)-f(y_1)\,\left(x_2+K(y_2)\right)}{f(y_1)\,g(y_2)-f(y_2)\,g(y_1)}\, g[q_Y(y_1,a,y_2)]
\\[20pt]
-K[q_Y(y_1,a,y_2)],
\end{array}$$
when $y_2\neq y_1$ and $\chi_a\left(x_1,y,x_2,y\right)=q_X(x_1,a,x_2)$ otherwise.
\end{proposition}

\begin{proof}
This is just Proposition \ref{General-Example} for the case $$h(\alpha,t,\beta)=\alpha\,f(t)+\beta\,g(t)-K(t).$$ With $U=V=X$, the system (\ref{system_h}) simply reads
$$\left(\begin{array}{cc}
f(y_1)\ g(y_1)\\[5pt]
f(y_2)\ g(y_2)
\end{array}
\right)
\left(\begin{array}{c}
\alpha\\[5pt]
\beta
\end{array}
\right)=
\left(\begin{array}{c}
x_1+K(y_1)\\[5pt]
x_2+K(y_2)
\end{array}
\right).
$$
\end{proof}

To illustrate this Proposition, Example \ref{ex5to9}(\ref{ex9}) can be generalized using on $Y$ (generalizing the set $I$) an arbitrary mobi space. Consider $h(\alpha,t, \beta)=\alpha\, f(t) + \beta$, in any set $X$ for which the next formula is well defined. Then, when $y_1\neq y_2$,
$$
\begin{array}{l}
q\left[(x_1,y_1),a,(x_2,y_2)\right]=\\[0.3mm]
(x_1+(x_2-x_1)\,\dfrac{f(q_Y(y_1,a,y_2))-f(y_1)}{f(y_2)-f(y_1)},q_Y(y_1,a,y_2)).
\end{array}
$$ 
When $y_1=y_2$, $q\left[(x_1,y_1),t,(x_2,y_1)\right]=(q_X(x_1,a,x_2),y_1)$ for any mobi operation $q_X$.

Even when the system of equations (\ref{system_h}) does not have a unique solution, then, in some cases, it is still possible to define a mobi-space. This will be illustrated with the formula for spherical linear interpolation giving geodesics on the $n$-sphere.

\section{Geodesics on the n-sphere}\label{sec.Slerp}

The purpose of this section is to show that a mobi space can be obtained using the geodesic curves on the $n$-sphere 
$$S^n=\{x\in\R^{n+1}\mid\langle x,x\rangle _E=1\}$$ 
and on one sheet of the two-sheeted hyperbolic $n$-space \cite{Hyperbolic}, as for instance 
$$H^n=\{x\in\R^{n+1}\mid \langle x,x\rangle _L=-1,x_1>0\}.$$ 
The notations $\langle\,,\rangle_E$ and $\langle\,,\rangle_L$ are used for the usual Euclidean and Lorentzian inner products, respectively. For the construction of the mobi operation for both cases at once, it is convenient to consider the family of functions
\begin{equation}\label{fslerp}
f(a)=\frac{e^{\alpha a}-e^{-\alpha a}}{2 \alpha}\ \textrm{and}\ g(a)=\frac{e^{\alpha a}+e^{-\alpha a}}{2},
\end{equation}
where $a\in \R$ and the parameter $\alpha$ is a non-zero complex number. These functions are real functions if and only if $\alpha$ is a real number or a pure imaginary number. In particular, we have that:
\begin{itemize}
\item $\alpha=1 \Rightarrow f(a)=\sinh a\ \textrm{and}\ g(a)=\cosh a$;
\item $\alpha=i \Rightarrow f(a)=\sin a\ \textrm{and}\ g(a)=\cos a$;
\item $\alpha\to 0 \Rightarrow f(a)=a\ \textrm{and}\ g(a)=1$.
\end{itemize}
For our purpose, we want to consider only real functions and therefore, for the rest of this section, it is understood that $\alpha$ is such that $f$ and $g$ are real.
In any case, the functions (\ref{fslerp}) verify the following properties:
\begin{eqnarray}
-\alpha^2 f^2(a)+g^2(a)&=&1\label{41}\\
f(a) g(b)+f(b) g(a)&=&f(a+b)\\
\alpha^2 f(a) f(b)+g(a) g(b)&=&g(a+b)\label{43}\\
f(-a)=-f(a)&,& g(-a)=g(a)\label{44}\\
f(0)=0&,& g(0)=1.\label{45}
\end{eqnarray}

In general terms, let us consider an interval $I\in\R$ containing $0$ where g is injective. Let $V$ be an inner product space, with the inner product denoted by $\langle\,,\rangle $, and a subspace $X\subseteq\{x\in V\mid \langle x,x\rangle =-\alpha^2\}$. Inner product here means a nondegenerate symmetric bilinear form. We are going to show that when there exists a unique function 
$$\theta:X\times X\to I$$ such that 
$-\alpha^2 g[\theta(x,y)]= \langle x,y\rangle $,
$X$ can be given the structure of a mobi space. For instance, if $X$ is $S^n$, $\theta$ may be defined as 
$$\theta(x,y)=\arccos(\langle x,y\rangle _E)\ \textrm{with}\ I=[0,\pi]$$ and if $X=H^n$, $\theta$ may be defined as
$$\theta(x,y)=\arccosh(-\langle x,y\rangle _L)\ \textrm{with}\ I=[0,+\infty[.$$ The expressions are similar for any pure imaginary or non-zero real number~$\alpha$. The next two Propositions show explicitly how to construct a mobi operation on $X$ using the functions $f$, $g$ and $\theta$. This construction is based on the spherical linear interpolation (Slerp) used in computer graphics \cite{Slerp}. The first proposition is for the cases where the geodesic between two points is unique which occur when the only zero of $f$, in $I$, is zero. This is what happens for $H^n$ but not for $S^n$ because $\sin(\pi)=0$. Nevertheless, the Proposition \ref{Slerp1} could still be applied to a portion of the n-sphere which does not contain antipodal points, such as for example $\{x\in\R^{n+1}\mid  \langle x,x\rangle _E=1, x_1>0\}$.

\begin{proposition}\label{Slerp1}
Consider the {\bf real} functions $f$ and $g$, of one real variable, verifying the properties {\normalfont (\ref{41})} to {\normalfont(\ref{45})} for some number $\alpha$. Suppose that $g$ is injective in an interval $I$ containing $0$ and that, for $a\in I$, we have 
$$f(a)=0\Longleftrightarrow a=0.$$ Let $(V,\langle\,,\rangle )$ be a real inner vector space and consider a subspace 
$$X\subseteq\{x\in V\mid \langle x,x\rangle =-\alpha^2\}.$$ If there exists a unique function $\theta:X\times X\to I$ such that 
\begin{equation}
\langle x,y\rangle =-\alpha^2 g[\theta(x,y)]
\end{equation}
and $\theta(x,y)=0\Longleftrightarrow y=x,$
then $(X,q)$ is a mobi space over the canonical mobi algebra where the ternary operation $q:X\times[0,1]\times X\to X$ is defined, for $x\neq y$, by
\begin{equation}\label{eqSlerp}
q(x,t,y)=\frac{f[\theta(x,y)\, (1-t)]}{f[\theta(x,y)]}\, x+ \frac{f[\theta(x,y)\, t]}{f[\theta(x,y)]}\, y,
\end{equation}
and, otherwise, by $q(x,t,x)=x$.
\end{proposition}
\begin{proof}
To simplify the presentation, we use the notation $\Omega\equiv\theta(x,y)$.
First, we have to prove that, when $x,y\in X$, $q(x,t,y)$ is still in $X$. The case $y=x$ is obvious. For $y\neq x$:
\begin{eqnarray*} 
\langle q(x,t,y),q(x,t,y)\rangle &=&\frac{f^2(\Omega (1-t))}{f^2(\Omega)} \langle x,x\rangle +\frac{f^2(\Omega t)}{f^2(\Omega)} \langle y,y\rangle \\
                            &+&2\frac{f(\Omega (1-t)) f(\Omega t)}{f^2(\Omega)} \langle x,y\rangle 
\end{eqnarray*}
\begin{eqnarray*}
&=&-\frac{\alpha^2}{f^2(\Omega)} (f^2(\Omega (1-t))+f^2(\Omega t)+2 f(\Omega (1-t)) f(\Omega t) g(\Omega))\\
&=&-\frac{\alpha^2}{f^2(\Omega)} (f(\Omega-\Omega t))(f(\Omega)g(\Omega t)+f(\Omega t) g(\Omega))+f^2(\Omega t))\\
&=&-\frac{\alpha^2}{f^2(\Omega)} (f^2(\Omega)g^2(\Omega t)+f^2(\Omega t)(1-g^2(\Omega))\\
&=&-\alpha^2 (g^2(\Omega t)-\alpha^2 f^2(\Omega t))=-\alpha^2
.\end{eqnarray*}
Now, as $X$ is a subspace containing $x$ and $y$ and $q(x,t,y)$ is a linear combination of $x$ and $y$, we conclude that it is also in $X$.
Axioms \ref{space_0}, \ref{space_1}, \ref{space_idem} of a mobi space are a direct consequence of the definition of $q$. To prove \ref{space_cancel} we will use the notation $\Omega'\equiv\theta(x,y')$. If $x\neq y$ and $x\neq y'$, we have that $q\left(x,\oh,y\right)=q\left(x,\oh,y'\right)$ implies
\begin{eqnarray}
\frac{f\left(\frac{\Omega}{2}\right)}{f(\Omega)} (x+y)=\frac{f\left(\frac{\Omega'}{2}\right)}{f(\Omega')} (x+y')\label{proofX4}
\end{eqnarray}
Applying the inner product with $x$ in both sides of this equation and using properties (\ref{41}) to (\ref{43}) in the form
$$
f(\Omega)=2 f\left(\frac{\Omega}{2}\right) g\left(\frac{\Omega}{2}\right)\ \textrm{and}\quad
1+g(\Omega)=2 g^2\left(\frac{\Omega}{2}\right),
$$
we get
\begin{eqnarray*}
&&\frac{1}{2g\left(\frac{\Omega}{2}\right)}\langle x,x+y\rangle =\frac{1}{2g\left(\frac{\Omega'}{2}\right)}\langle x,x+y'\rangle \\
&\Rightarrow&
\frac{1}{2g\left(\frac{\Omega}{2}\right)}(-\alpha^2-\alpha^2g(\Omega))=\frac{1}{2g\left(\frac{\Omega'}{2}\right)}(-\alpha^2-\alpha^2g(\Omega'))\\
&\Rightarrow&
g\left(\dfrac{\Omega}{2}\right)=g\left(\dfrac{\Omega'}{2}\right)\neq 0.
\end{eqnarray*}
Going back to (\ref{proofX4}) with this result, we conclude $y=y'$. If $x=y$ and $x\neq y'$, then $q\left(x,\oh,y\right)\neq q\left(x,\oh,y'\right)$. Indeed,  $q\left(x,\oh,x\right)= q\left(x,\oh,y'\right)$ would imply
$$\begin{array}{l}
x=\frac{f\left(\frac{\Omega'}{2}\right)}{f(\Omega')} (x+y')\Rightarrow \langle x,x\rangle =\frac{1}{2 g\left(\frac{\Omega'}{2}\right)}\,\langle x,x+y'\rangle \\[10pt]
\Rightarrow 1=\frac{1}{2 g\left(\frac{\Omega'}{2}\right)}\,(1+g(\Omega'))\Rightarrow g\left(\frac{\Omega'}{2}\right)=1\Rightarrow y'=x,
\end{array}$$
in contradiction with the hypothesis. The case $x\neq y$ and $x=y'$ is similar. Obviously \ref{space_cancel} is also verified if $x=y$ and $x=y'$.
The proof of \ref{space_homo} begin with the observation that:
\begin{eqnarray}
g[\theta(q(x,a,y),q(x,c,y))]=g[\theta(x,y)(c-a)].\label{qaqc}
\end{eqnarray}
Indeed, beginning with the left-hand side of (\ref{qaqc}), if $y\neq x$:
\begin{eqnarray*}
&&-\frac{1}{\alpha^2} \langle \frac{f(\Omega (1-a))}{f(\Omega)} x+ \frac{f(\Omega a)}{f(\Omega)} y,
\frac{f(\Omega (1-c))}{f(\Omega)} x+ \frac{f(\Omega c)}{f(\Omega)} y\rangle \\
&=&\frac{f(\Omega (1-a))f(\Omega (1-c))}{f^2(\Omega)}+\frac{f(\Omega a)f(\Omega c)}{f^2(\Omega)}\\
&+&\left(\frac{f(\Omega (1-a))f(\Omega c)+f(\Omega (1-c))f(\Omega a)}{f^2(\Omega)}\right)g(\Omega)\\
&=&\frac{f^2(\Omega)g(\Omega a)g(\Omega c)-g^2(\Omega)f(\Omega a)f(\Omega c)+f(\Omega a)f(\Omega c)}{f^2(\Omega)}\\
&=&g(\Omega a)g(\Omega c)-\alpha^2f(\Omega a)f(\Omega c)\\
&=&g(\Omega a-\Omega c)=g(\Omega c-\Omega a)
\end{eqnarray*}
If $y=x$, then
\begin{eqnarray*}
g[\theta(q(x,a,y),q(x,c,y))]&=&-\frac{1}{\alpha^2} \langle q(x,a,y),q(x,c,y)\rangle \\
                            &=&-\frac{1}{\alpha^2} \langle x,x\rangle \\
                            &=& 1=g(0)=g[\theta(x,y)(c-a)].
\end{eqnarray*}
Now, because $a,c\in[0,1]$ and $\Omega\in I$ imply $\Omega |c-a|\in I$, we can conclude, since $g$ is injective in I, that 
\begin{equation}
|\theta(q(x,a,y),q(x,c,y))|=|\theta(x,y)(c-a)|=|\Omega (c-a)|.
\end{equation}
By (\ref{44}), this also imply that, for any $b\in [0,1]$, the following relation is true:
\begin{equation*}
\frac{f[\theta(q(x,a,y),q(x,c,y))\, b]}{f[\theta(q(x,a,y),q(x,c,y))]}=\frac{f[\Omega (c-a)\,b]}{f[\Omega (c-a)]}.
\end{equation*}
With these results, we are able to prove \ref{space_homo}. For simplification, we use the notation $\hat{c}\equiv c-a$. First, for $q(x,a,y)\neq q(x,c,y)$ and $x\neq y$:
\begin{eqnarray*}
&&q[q(x,a,y),b,q(x,c,y)]\\
&=&\frac{f[\Omega\,\hat{c}(1-b)]}{f[\Omega\,\hat{c}]}\,q(x,a,y)+\frac{f[\Omega\,\hat{c}\,b]}{f[\Omega\,\hat{c}]}\,q(x,\hat{c}+a,y)\\
&=&\frac{f[\Omega\,\hat{c}-\Omega\,\hat{c}\,b]f[\Omega(1-a)]+f[\Omega\,\hat{c}\,b]f[\Omega(1-a)-\Omega\,\hat{c}]}{f[\Omega\,\hat{c}]f(\Omega)}\,x\\
&\ +&\frac{f[\Omega\,\hat{c}-\Omega\,\hat{c}\,b]f[\Omega a]+f[\Omega\,\hat{c}\,b]f[\Omega\,a+\Omega\,\hat{c}]}{f[\Omega\,\hat{c}]f(\Omega)}\,y\\
&=&\frac{g[\Omega\,\hat{c}\,b]f[\Omega(1-a)]-f[\Omega\,\hat{c}\,b]g[\Omega(1-a)]}{f(\Omega)}\,x\\
&+&\frac{g[\Omega\,\hat{c}\,b]f[\Omega\,a]+f[\Omega\,\hat{c}\,b]g[\Omega\,a]}{f(\Omega)}\,y\\
&=&\frac{f[\Omega(1-a-\hat{c}\,b)]}{f(\Omega)}\,x+\frac{f[\Omega(a+\hat{c}\,b)]}{f(\Omega)}\,y\\
&=&q[x,a+(c-a)b,y].
\end{eqnarray*}
If $q(x,a,y)=q(x,c,y)$, on the one hand $q(q(x,a,y),b,q(x,c,y))=q(x,a,y)$. On the other hand, from (\ref{qaqc}), we conclude that $x=y$ or $c=a$ and in both cases $q(x,a+b(c-a),y)=q(x,a,y)$.
\end{proof}

Before going to Proposition \ref{Slerp2} that will explain how we can still get a mobi space out of a {\it Slerp type} formula on the $n$-sphere despite the fact that geodesics between antipodal points are not unique, let us take a closer look to the formula (\ref{eqSlerp}) in the case of $S^n$. This formula just gives the intersection between $S^n$ and a plane that contains the origin and the points $x$ and $y$, when $x$ and $y$ are not collinear. Starting at $x$ when $t=0$, a particle that goes to $y$ on that plane at constant speed will be, at an instant $t\in[0,1]$ at 
\begin{equation}\label{z1}
q(x,t,y)= \cos(\Omega\,t)\,x + \sin(\Omega\,t)\,z
\end{equation}
where 
\begin{equation}\label{z2}
\cos(\Omega)\,x+ \sin(\Omega)\,z=y.
\end{equation}
Because $\langle x,x\rangle _E=\langle y,y\rangle _E=1$ and $\Omega\equiv\theta(x,y)=\arccos\langle x,y\rangle _E$, we have that $\langle x,z\rangle _E=0$ and $\langle z,z\rangle _E=1$. When $\sin(\Omega)\neq 0$, we can just solve (\ref{z2}) to obtain $z$ and then (\ref{z1}) reads as expected:
$$q(x,t,y)=\frac{\sin[\Omega\,(1-t)]}{\sin(\Omega)}\,x+\frac{\sin[\Omega\,t]}{\sin(\Omega)}\,y.$$
When $\Omega=0$, which means $x=y$, there is no journey to make: $q(x,t,x)=x$. When $\Omega=\pi$, which means $y=-x$, we have to choose the plane we wish to travel on. Equivalently, we have to chose the direction $v(x)\in\R^{n+1}$ we want to be playing the role of $z$. Of course, we still need $\langle x,v(x)\rangle _E=0$ and $\langle v(x),v(x)\rangle _E=1$. There is one more condition: to get a mobi space, we also need $v$ to be an even map because $q(x,t,-x)=q(-x,1-t,x)$ (property \ref{Y1}) which means that in a round trip, the going and the return must be done on the same path.

\begin{proposition}\label{Slerp2}
Consider the euclidean $n$-sphere $$X=\{x\in\R^{n+1}\mid\,\langle x,x\rangle =1\},$$ a map $v:X\to X$ such that
$$
v(-x)=v(x)\ \textrm{and}\ \langle x,v(x)\rangle =0,
$$
and the map $\theta:X\times X\to [0,\pi]$ defined by $$\theta(x,y)=\arccos(\langle x,y\rangle ).$$
With $q:X\times[0,1]\times X\to X$ defined by
$$
\begin{array}{l}
q(x,t,y)=\\[3mm]
\left\{ 
\begin{array}{ccl}
\frac{\sin[\theta(x,y)\, (1-t)]}{\sin[\theta(x,y)]}\, x+ \frac{\sin[\theta(x,y)\, t]}{\sin[\theta(x,y)]}\, y&,if &\theta(x,y)\in ]0,\pi[\\[3mm]
\cos[\theta(x,y)\, t]\,x+\sin[\theta(x,y)\, t]\,v(x) &,if &\theta(x,y)\in\{0,\pi\}
\end{array}
\right.,
\end{array}$$
$(X,q)$ is a mobi space over the canonical mobi algebra.
\end{proposition}
\begin{proof}
Most of the proof is the same as the proof of Proposition \ref{Slerp1}. We just have to consider the extra case where $y=-x$ corresponding to \mbox{$\Omega\equiv\theta(x,y)=\pi$}. 
When $\Omega=\pi$, we have $q(x,0,y)=\cos(0)x+\sin(0) v(x)=x$ and $q(x,1,y)=\cos(\pi)x+\sin(\pi) v(x)=-x=y$, so Axioms \ref{space_0}, \ref{space_1} and \ref{space_idem} of a mobi space are verified. Regarding \ref{space_cancel}, when $\Omega=\pi$ and $\Omega'\in]0,\pi[$, we have that $q\left(x,\oh,y\right)\neq q\left(x,\oh,y'\right)$. Indeed $q\left(x,\oh,y\right)=q\left(x,\oh,y'\right)$ implies
$$
v(x)=\frac{f\left(\frac{\Omega'}{2}\right)}{f(\Omega')} (x+y')\Rightarrow
0=g\left(\frac{\Omega'}{2}\right)\Rightarrow \Omega'=\pi,
$$
in contradiction with $\Omega'\in]0,\pi[$.
The case $\Omega=\pi$ and $\Omega'=0$ is also incompatible with $q\left(x,\oh,y\right)=q\left(x,\oh,y'\right)$ because $v(x)\neq x$. Interchanging $y$ and $y'$ in the previous situations gives similar results. The case $\Omega=\pi$ and $\Omega'=\pi$ implies $y=-x=y'$, therefore \ref{space_cancel} is verified. Regarding \ref{space_homo}, we first observe that (\ref{qaqc}) is valid for all $x, y\in X$. Indeed, if $y=-x$, then
\begin{eqnarray*}
&&\cos[\theta(q(x,a,y),q(x,c,y))]= \langle q(x,a,y),q(x,c,y)\rangle \\
                            &=&\langle \cos(\pi a)\,x+\sin(\pi a)\,v(x),\cos(\pi c)\,x+\sin(\pi c)\,v(x)\rangle \\
                            &=& \cos(\pi a)\cos(\pi c)+\sin(\pi a)\sin(\pi c)\\
                            &=& \cos[\pi(a-c)].
\end{eqnarray*}
So, we have that, $\forall x,y\in X$:
\begin{equation}\label{qaqc2}
\theta(q(x,a,y),q(x,c,y))=\theta(x,y) |c-a|.
\end{equation}
From equation (\ref{qaqc2}), we conclude that $\theta(q(x,a,y),q(x,c,y))=\pi$ if and only if $\Omega=\pi$ and $|c-a|=1$ and that $\theta(q(x,a,y),q(x,c,y))=0$ if and only if $\Omega=0$ or $c=a$. Therefore, besides the cases already proved in Proposition \ref{Slerp1}, we have to consider the following four situations: 
\begin{enumerate}
\item $\theta(q(x,a,y),q(x,c,y))=\pi$, $\Omega=\pi$ and
\begin{enumerate} 
\item $c=0$, $a=1$
\item $c=1$, $a=0$
\end{enumerate}
\item $\theta(q(x,a,y),q(x,c,y))=0$, $\Omega=\pi$ and $c=a$
\item $\theta(q(x,a,y),q(x,c,y))\in]0,\pi[$, $\Omega=\pi$, $c\neq a$ and $|c-a|\neq 1$.
\end{enumerate}
For the situation (1a):
\begin{eqnarray*}
q[q(x,a,y),b,q(x,c,y)]&=&q(-x,b,x)=\cos(\pi b) (-x)+\sin(\pi b)\, v(-x)\\
q[x,a+b(c-a),y]&=&q(x,1-b,-x)\\
               &=&\cos(\pi-\pi b)\, x+\sin(\pi-\pi b)\, v(x)\\
               &=&-\cos(\pi b)\, x+\sin(\pi b)\, v(x).
\end{eqnarray*}
The Axiom \ref{space_homo} is ensured through the hypothesis $v(-x)=v(x)$.
For the situation (1b):
\begin{eqnarray*}
q[q(x,a,y),b,q(x,c,y)]=q(x,b,-x)&=&\cos(\pi b)\, x+\sin(\pi b)\, v(x)\\
q[x,a+b(c-a),y]=q(x,b,-x)&=&\cos(\pi b)\, x+\sin(\pi b)\, v(x).
\end{eqnarray*}
In situation (2), $c=a$ and \ref{space_idem} implies \ref{space_homo}.
Using $\hat{c}\equiv c-a$, we have for the situation (3):
\begin{eqnarray*}
&&q[q(x,a,y),b,q(x,c,y)]\\
&=&\frac{\sin[\pi (c-a)(1-b)]}{\sin[\pi (c-a)]}\left(\cos(\pi a)\,x+\sin(\pi a)\,v(x)\right)\\
&+&\frac{\sin[\pi (c-a)b]}{\sin[\pi (c-a)]}\left(\cos(\pi c)\,x+\sin(\pi c)\,v(x)\right)\\
&=&\frac{\sin[\pi \hat{c}(1-b)]\cos(\pi a)+\sin[\pi \hat{c}\,b]\cos(\pi (\hat{c}+a))}{\sin[\pi\, \hat{c}]}\,x\\
&+&\frac{\sin[\pi \hat{c}(1-b)]\sin(\pi a)+\sin[\pi \hat{c}\,b]\sin(\pi (\hat{c}+a))}{\sin[\pi\, \hat{c}]}\,v(x)\\
&=&\cos[\pi\, \hat{c}\,b]\cos(\pi a)-\sin[\pi\, \hat{c}\,b]\sin[\pi a]\,x\\
&+&\cos[\pi\, \hat{c}\,b]\sin(\pi a)+\sin[\pi\, \hat{c}\,b]\cos[\pi a]\,v(x)\\
&=&\cos[\pi(a+\hat{c}\,b)]\,x+\sin[\pi(a+\hat{c}\,b)]\,v(x)\\
&=&q(x,a+b(c-a),y)
\end{eqnarray*}
\end{proof}
To finish this section, we present three examples of the map $v$ used in Proposition \ref{Slerp2}. 
First, consider the $1$-sphere i.e. the circle. We can choose to move between antipodal points in the anticlockwise direction when starting somewhere at the top of the circle and in the clockwise direction when starting at the bottom. More specifically, if $x=(\cos\theta,\sin\theta), \theta\in[0,2\pi[$, then $v$ is defined as
$$
v(x)=\left\{
\begin{array}{lcl}
(-\sin\theta,\cos\theta)&if& \theta\in[0,\pi[\\
(\sin\theta,-\cos\theta)&if& \theta\in[\pi,2\pi[
\end{array}
\right..
$$
Secondly, let us choose to connect two antipodal points on $S^2$, different from the poles, through the north pole and link the poles (on the z-axis) through the positive x-axis. This gives the following choice for $v$, considering $x_1^2+x_2^2+x_3^2=1$:
$$v(x_1,x_2,x_3)=\left\{
\begin{array}{ccl}
\dfrac{(-x_1 x_3,\, -x_2 x_3,\, 1-x_3^2)}{\sqrt{1-x_3^2}}&if& x_3\neq \pm 1\\[5mm]
(1,0,0) &if& x_3=\pm 1
\end{array}
\right..
$$
As a third example, consider the $2$-sphere parametrized in spherical coordinates as:
$$\{(\sin\varphi \cos\theta,\sin\varphi \sin\theta, \cos\varphi), (\theta,\varphi)\in([0,2\pi[\times]0,\pi[)\cup(0,0)\cup(0,\pi)\}.$$
A possible map $v$ is the following:
\begin{eqnarray*}
&&v(\sin\varphi \cos\theta,\sin\varphi \sin\theta, \cos\varphi)\\
&=&\left\{
\begin{array}{ccl}
(-\sin\theta,\cos\theta,0)&if& \varphi\in\left[0,\frac{\pi}{2}\right[ \textrm{or} \left(\varphi=\frac{\pi}{2}\ 
                                 , \theta\in[0,\pi[\right)\\[2mm]
(\sin\theta,-\cos\theta,0)&if& \varphi\in\left]\frac{\pi}{2},\pi\right] \textrm{or} \left(\varphi=\frac{\pi}{2}\ 
                                 , \theta\in[\pi,2\pi[\right)
\end{array}
\right..
\end{eqnarray*}
In this example, $v$ is on the equator in a plane rotated $\frac{\pi}{2}$ around the z-axis from the meridian of $x$. The choice $\theta=0$ for the poles connects them through the positive y-axis. The other antipodal points are connected through a path that stays between the parallels of the two points, with an arbitrary choice for antipodal points on the equator.

\section{Conclusion}

We have introduced a new algebraic structure which captures some features of geodesic paths. The particular affine case was studied. The example of geodesics on the $n$-sphere was deduced from the formula  Slerp (spherical linear interpolation). Other interesting lines of study include the connection with affine geometry \cite{bourn_geometrie,Kock2,Kock3} or the geometry of geodesics \cite{busemann-1943,busemann-1955}. The presence of an operation $x\oplus y=q(x,\muu,y)$ admitting cancellation, together with the property $x\oplus y=y\oplus x$, tells us that the category of mobi spaces is  a weakly Mal'tsev category \cite{NMF2012,NMF.08}. This was in fact the starting point that originated our investigation on mobi spaces. Further studies on spaces in which geodesics are not necessarily unique (such as the sphere) are worthwhile pursuing to better understand the importance of the choices that have to be made. This and other topics, such as the development of a homology theory for mobi spaces, will be investigated in future work.

\section*{Acknowledgement}

We thank Anders Kock 
for his helpful comments and suggestions to improve the text.


\end{document}